\documentclass[12pt]{amsart}
\usepackage{amssymb,amsmath,amsthm}
\usepackage{tikz}

\usetikzlibrary{arrows,shapes,automata,backgrounds,decorations,petri,positioning}

\theoremstyle{plain}
\newtheorem{thm}{Theorem}[section]
\newtheorem{lem}[thm]{Lemma}

\newtheorem{cor}[thm]{Corollary}

\theoremstyle{definition}
\newtheorem{defn}[thm]{Definition}

\newtheorem{ex}[thm]{Example}

\theoremstyle{plain}
\newtheorem{sthm}{Theorem}[subsection]
\newtheorem{slem}[sthm]{Lemma}

\newtheorem{scor}[sthm]{Corollary}

\newtheorem{sprop}[sthm]{Proposition}
\theoremstyle{definition}
\newtheorem{sdefn}[sthm]{Definition}

\newtheorem{sex}[sthm]{Example}

\newcommand{\mf}[1]{\mbox{$\mathfrak #1$}}
\newcommand{\ol}[1]{\mbox{$\overline #1$}}
\newcommand{\supp}{\textnormal{\textsf{supp}}}
\newcommand{\rep}{\textnormal{\textsf{rep}}}
\newcommand{\patt}{\textnormal{\textsf{[321;3412]}}}
\renewcommand{\repeat}{\textnormal{\textsf{new-rep}}}
\newcommand{\p}{\textnormal{\textsf{[321]}}}
\newcommand{\q}{\textnormal{\textsf{[3412]}}}
\newcommand{\pattern}{\ensuremath{\mathfrak{p}}}

\title{Repetition in reduced decompositions}

\author{Bridget Eileen Tenner}
\address{Department of Mathematical Sciences, DePaul University, Chicago, IL 60614}
\email{bridget@math.depaul.edu}
\thanks{Research partially supported by a DePaul University Faculty Summer Research Grant.}


\begin{document}

\begin{abstract}
Given a permutation $w$, we show that the number of repeated letters in a reduced decomposition of $w$ is always less than or equal to the number of 321- and 3412-patterns appearing in $w$.  Moreover, we prove bijectively that the two quantities are equal if and only if $w$ avoids the ten patterns 4321, 34512, 45123, 35412, 43512, 45132, 45213, 53412, 45312, and 45231.\\

\end{abstract}

\maketitle

\section{Introduction}

Permutations can be described in a variety of ways, including as a product of simple reflections and in one-line notation.  These two were studied extensively by the author in \cite{rdpp}, and a means for translating properties of one presentation into properties of the other was given.  The first of these presentations is most relevant to the generalized setting of Coxeter groups and the Bruhat order.  There is a rich literature studying various properties of reduced decompositions, including \cite{bjorner-brenti} and \cite{stanley}.  The second of these presentations, one-line notation, is primarily useful when discussing the notion of permutation patterns.  This topic originated in work of Rodica Simion and Frank Schmidt \cite{simion-schmidt}, and has become a popular subfield of combinatorics.

Given any permutation $w$, one can calculate its length, and one can also calculate the number of distinct simple reflections that appear in any reduced decomposition of $w$.  The difference between these two quantities, denoted $\rep(w)$ in this paper, would thus count the number of repeated letters in any reduced decomposition of $w$.  These statistics are readily computed from the presentation of a permutation as a product of simple reflections.

When written in one-line notation, one often looks at the patterns in (or not in) a permutation.  In particular, one can count the number of distinct $321$- and $3412$-patterns in a permutation $w$, and this total will be denoted $\patt(w)$ here.

It was shown in previous work by the author that $\rep(w) = 0$ if and only if $\patt(w) = 0$ \cite{patt-bru}.  Additionally, Daniel Daly shows that $\rep(w) = 1$ if and only if $\patt(w) = 1$ \cite{daly}.  Other than these results, not much has been known about the quantity or type of repetition that might occur within a reduced decomposition of a given permutation.

The ideal conclusion based on the results of \cite{patt-bru} and \cite{daly}, that $\rep(w)$ and $\patt(w)$ would always be equal, is not actually the case, as can be seen with $\rep(4321) = 3$ and $\patt(4321) = 4$.  However, the main result of this paper (Theorem~\ref{thm:main}) is that $\rep(w)$ is always less than or equal to $\patt(w)$, and the two quantities are equal exactly when $w$ avoids each of the patterns
$$\{4321, 34512, 45123, 35412, 43512, 45132, 45213, 53412, 45312, 45231\}.$$
Moreover, in Corollary~\ref{cor:bound}, we give a crude lower bound on the difference $\patt(w) - \rep(w)$ when $w$ contains some of the patterns listed above.

In Section~\ref{section:definitions} of the paper, we introduce the necessary objects and terminology for this work.  Section~\ref{section:main theorem} suggests the relevance of the ten patterns listed above and states the main theorem, while the proof of this theorem is spread over Sections~\ref{section:preliminaries} and~\ref{section:proof}.

\section{Definitions}\label{section:definitions}

This section summarizes the primary objects studied in this work.  More background on this material can be found in \cite{bjorner-brenti} and \cite{macdonald}.

Let $\mathfrak{S}_n$ be the symmetric group on $n$ elements.  The group $\mf{S}_n$ is generated by the simple reflections (also called adjacent transpositions) $\{s_1,\ldots,s_{n-1}\}$, where $s_i$ is the permutation interchanging $i$ and $i+1$, and fixing all other elements.  These permutations satisfy the Coxeter relations
\begin{eqnarray*}
& s_i^2 = 1 &\text{ for all } i,\\
& s_is_j = s_js_i &\text{ if } |i-j| > 1, \text{ and}\\
& s_is_{i+1}s_i = s_{i+1}s_is_{i+1} &\text{ for } 1 \le i \le n-2.
\end{eqnarray*}

We adopt the custom that $s_iw$ interchanges the positions of the values $i$ and $i+1$ in the one-line notation of $w$, and $ws_i$ interchanges the values in positions $i$ and $i+1$ in the one-line notation of $w$.

A permutation $w \in \mf{S}_n$ can also be written in one-line notation as $w = w(1)w(2) \cdots w(n)$.

\begin{ex}
The permutation $3241 \in \mf{S}_4$ maps $1$ to $3$, $2$ to itself, $3$ to $4$, and $4$ to $1$.
\end{ex}

We have now described two substantially different presentations for permutations: products of simple reflections and one-line notation.  A means of translating between these two, and  of inferring properties of one from properties of the other, was given in \cite{rdpp}.

\begin{defn}
If $w = s_{i_1}\cdots s_{i_{\ell(w)}}$ where $\ell(w)$ is minimal, then $s_{i_1}\cdots s_{i_{\ell(w)}}$ is a \emph{reduced decomposition} of $w$.  This $\ell(w)$ is the \emph{length} of $w$.
\end{defn} 

The set of reduced decompositions of a permutation has been studied from several viewpoints, including connections to Young tableaux as described in \cite{stanley}.  In this paper, we will study repetition among the letters in a reduced decomposition of a permutation.  To that end, we make the following definition.

\begin{defn}
Given a permutation $w$, the \emph{support} of $w$ is the set $\supp(w)$ of distinct letters appearing in a reduced decomposition of $w$.
\end{defn}

It is important to clarify why this definition is sound.

\begin{lem}\label{lem:supp well defined}
The set $\supp(w)$ is well defined.
\end{lem}

\begin{proof}
We must prove that the set of letters in a reduced decomposition of a permutation is independent of the particular reduced decomposition chosen as a representative.  Any reduced decomposition of $w$ can be obtained from any other by a series of Coxeter relations (\cite{matsumoto} and \cite{tits}, independently).  These do not change the underlying set of distinct letters in the reduced decomposition, so the set $\supp(w)$ is well defined.  That is, given any reduced decomposition $w = s_{i_1}\cdots s_{i_{\ell}}$,
$$\supp(w) = \{s_{i_1}, \ldots, s_{i_{\ell}}\}.$$
\end{proof}

\begin{ex}
Let $w = 32154 \in \mf{S}_5$.  One reduced decomposition for $w$ is $s_2s_1s_2s_4$, so $\supp(w) = \{s_1,s_2,s_4\}$.
Note that $s_2s_1s_4s_2$ and $s_1s_2s_1s_4$ are also reduced decompositions for $w$, and they each yield the same set $\supp(w)$.
\end{ex}

The following statistics will be crucial in our proof of the main theorem.

\begin{defn}\label{defn:max and min}
Fix $w \in \mf{S}_n$ and $k \in \{1, \ldots, n-1\}$.  Let
$$M_k(w) = \max\{w(1), \ldots, w(k)\}$$
and
$$m_k(w) = \min\{w(k+1), \ldots, w(n)\}.$$
\end{defn}

\begin{lem}
For any $w \in \mf{S}_n$, the values of $M_k(w)$ satisfy
$$M_1(w) \le M_2(w) \le M_3(w) \le \cdots \le M_{n-1}(w),$$
and the values of $m_k(w)$ satisfy
$$m_1(w) \le m_2(w) \le m_3(w) \le \cdots \le m_{n-1}(w).$$
We have strict inequality $M_k(w) < M_{k+1}(w)$ exactly when $w(k+1) > M_k(w)$, and $m_k(w) < m_{k+1}(w)$ exactly when $w(k+1)<m_{k+1}(w)$.
\end{lem}

\begin{proof}
These inequalities follow immediately from the definitions of $M_k(w)$ and $m_k(w)$.
\end{proof}

The next lemma is a consequence of the definition of the support of a permutation.

\begin{lem}\label{lem:support implications}
Fix a permutation $w \in \mf{S}_n$.  The following statements are equivalent:
\begin{itemize}
\item $s_k \in \supp(w)$,
\item $\{w(1), \ldots, w(k)\} \neq \{1,\ldots,k\}$,
\item $\{w(k+1), \ldots, w(n)\} \neq \{k+1, \ldots, n\}$,
\item $M_k(w) > k$,
\item $m_k(w) < k+1$, and
\item $M_k(w) > m_k(w)$.
\end{itemize}
\end{lem}

\begin{proof}
Suppose that $s_k \in \supp(w)$.  This means that $s_k$ appears at least once in each reduced decomposition of $w$, which means that there is some inversion $w(i) > w(j)$ in $w$, where $i \le k < j$.  Thus the set $\{w(1), \ldots, w(k)\}$ cannot equal $\{1,\ldots,k\}$, and, equivalently, the set $\{w(k+1), \ldots, w(n)\}$ cannot equal $\{k+1, \ldots, n\}$.  Also equivalently, the set $\{w(1), \ldots, w(k)\}$ contains an element larger than $k$, and, equivalently, the set $\{w(k+1), \ldots, w(n)\}$ contains an element less than $k+1$.

If, on the other hand, $s_k \not\in \supp(w)$, then there is no inversion such as described in the previous paragraph.  Therefore $w(1)\cdots w(k)$ is a permutation of $\{1,\ldots,k\}$ and $w(k+1)\cdots w(n)$ is a permutation of $\{k+1,\ldots,n\}$.  Thus $M_k(w) = k$ and $m_k(w) = k+1$.
\end{proof}

In this paper, we will study the relationship between two statistics of a permutation.  The first of these is related to the support of a permutation.

\begin{defn}
Given a permutation $w$, let $\rep(w)$ be the quantity
\begin{equation}\label{eqn:rep defn}
\rep(w) = \ell(w) - |\supp(w)|.
\end{equation}
\end{defn}

This quantity is so named because it counts the number of simple reflections in a reduced decomposition of $w$, when reading from one end to the other, which repeat previously seen letters.  The fact that this latter description is well defined may not be immediately obvious, given that a permutation may have more than one reduced decomposition.  However, this does not affect $\supp(w)$, as shown by Lemma~\ref{lem:supp well defined}, and so $\rep(w)$ is well defined, by equation~\eqref{eqn:rep defn}.

\begin{ex}\label{ex:35412 reps}
Let $w = 35412$, where $\ell(w) = 7$ and $\supp(w) = \{s_1,s_2,s_3,s_4\}$.  Thus $\rep(w) = 7-4 = 3$.  Relatedly, one reduced decomposition for $w$ is $s_2s_1s_3s_2s_4s_3s_2$, and reading from left to right we encounter the repeated simple reflections which are marked in
$$s_2s_1s_3 \fbox{$s_2$} s_4 \fbox{$s_3$} \fbox{$s_2$}.$$
There are three such letters, so $\rep(w) = 3$.
\end{ex}

The other statistic we will consider relates to permutation patterns.

\begin{defn}
Let $w \in \mf{S}_n$ and $p \in \mf{S}_k$ for $k \le n$.  The permutation $w$ \emph{contains the pattern $p$} if there exist $i_1 < \cdots < i_k$ such that $w(i_1) \cdots w(i_k)$ is in the same relative order as $p(1) \cdots p(k)$, in which case $w(i_1) \cdots w(i_k)$ is an \emph{occurrence} of $p$ in $w$.  For notational clarity, we will sometimes denote this pattern by $\{w(i_1), \ldots, w(i_k)\}$.  If $N = \max\{w(i_1), \ldots, w(i_k)\}$, then this $w(i_1) \cdots w(i_k)$ is an \emph{$N$-occurrence} of $p$.  If $w$ does not contain $p$, then $w$ \emph{avoids} $p$, or is \emph{$p$-avoiding}.
\end{defn}

The set of all occurrences of a pattern $p$ in a permutation $w$ can be partitioned by the largest letter appearing in the occurrence:
$$\{\text{occurrences of } p \text{ in } w \} = \bigsqcup_N \{N\text{-occurrences of } p \text{ in } w \}.$$

\begin{ex}
Continuing Example~\ref{ex:35412 reps}, there are two occurrences of $3412$ in $w$: $3512$ and $3412$.  The first of these is a $5$-occurrence, and the second is a $4$-occurrence.  The permutation $w$ is $123$-avoiding because it has no increasing subsequence of length $3$.
\end{ex}

There is much interest in enumeration related to permutation patterns (see, for example, \cite{claesson-kitaev, marcus-tardos, simion-schmidt}).  The portion of this scholarship relevant to the current work is the enumeration of occurrences of a pattern $p$ appearing in a permutation $w$.

\begin{defn}
Given a permutation $w$ and a pattern $p$, let $\textsf{[p]}_N(w)$ denote the number of $N$-occurrences of $p$ in $w$.  Let
$$\textsf{[p]}(w) = \sum_N \textsf{[p]}_N(w)$$
be the total number of occurrences of $p$ in $w$.
\end{defn}

\begin{ex}
Continuing Example~\ref{ex:35412 reps}, we have $\p_5(w) = 2$ and $\p_i(w) = 0$ for all $i\neq 5$.  Also, $\q_4(w) = \q_5(w) = 1$, and $\q_i(w) = 0$ otherwise.
\end{ex}

For reasons that will be suggested by Theorem~\ref{thm:0 and 1}, we are most concerned with the patterns $321$ and $3412$, and we will count the number of distinct occurrences of these patterns.

\begin{defn}
Given a permutation $w$, and a positive integer $N$, let
$$\patt_N(w) = \p_N(w) + \q_N(w).$$
Let $\patt(w)$ be the quantity
\begin{eqnarray}\label{eqn:patt definition}
\patt(w) &=& \p(w) + \q(w)\\
\nonumber &=& \sum_N \patt_N(w).
\end{eqnarray}
\end{defn}

\begin{ex}\label{ex:35412 patts}
Continuing Example~\ref{ex:35412 reps}, let us calculate $\patt(w)$.  The distinct occurrences of $321$ in $w$ are $\{541,542\}$, and the distinct occurrences of $3412$ in $w$ are $\{3512,3412\}$.  Thus
$$\begin{array}{r@{\ =\ }l@{\ =\ }l}
\patt_4(w) & 0 + 1 & 1,\\
\patt_5(w) & 2 + 1 & 3, \text{ and}\\
\patt(w) & 2+2 & 1 + 3 =4.
\end{array}$$
\end{ex}

Using the notation defined above, the following results were shown previously, the first by the author and the second  by Daniel Daly.

\begin{thm}[\cite{patt-bru} and \cite{daly}]\label{thm:0 and 1}
For any permutation $w$, 
\begin{enumerate}\renewcommand{\labelenumi}{(\alph{enumi})}
\item $\rep(w) = 0$ if and only if $\patt(w) = 0$, and
\item $\rep(w) = 1$ if and only if $\patt(w) = 1$.
\end{enumerate}
\end{thm}

Theorem~\ref{thm:0 and 1} gives a clear indication that the statistic $\rep$ is related to whether a permutation contains the patterns $321$ or $3412$.  This arises from the previously mentioned work by the author in \cite{rdpp}, relating patterns (and hence the one-line  presentation of a permutation) with the presentation of a permutation as a product of simple reflections.

The statistics $\rep$ and $\patt$ are not always equal, as shown by Examples~\ref{ex:35412 reps} and~\ref{ex:35412 patts}:
$$\rep(35412) < \patt(35412).$$

In this paper, we will show that $\rep(w)$ never exceeds $\patt(w)$, and we will characterize equality of the two quantities by pattern avoidance.

\section{The main theorem}\label{section:main theorem}

Rather surprisingly, the potential equality of the statistics $\rep$ and $\patt$ mentioned at the end of the last section depends solely on the avoidance of ten patterns, the set of which we will denote $\Phi$.

\begin{defn}
Let
$$\Phi = \{4321, 34512, 45123, 35412, 43512, 45132, 45213, 53412, 45312, 45231\} \subset (\mf{S}_4 \cup \mf{S}_5).$$
\end{defn}

Note that the subset $\{34512, 45123, 35412, 43512, 45132, 45213, 53412, 45312, 45231\} \subset \Phi$ can be expressed as the single marked mesh pattern
\begin{center}
\begin{tikzpicture}[scale=.5]
\fill[black!20] (.1,4.5) arc(180:0:.4) -- (.9,3.9) -- (2.5,3.9) arc(90:0:.4) -- (2.9,1.9) -- (4.5,1.9) arc(90:0:.4) -- (4.9,.5) arc(0:-180:.4) -- (4.1,1.1) -- (2.5,1.1) arc(270:180:.4) -- (2.1,3.1) -- (.5,3.1) arc(270:180:.4) -- (.1,4.5);
\draw (.1,4.5) arc(180:0:.4) -- (.9,3.9) -- (2.5,3.9) arc(90:0:.4) -- (2.9,1.9) -- (4.5,1.9) arc(90:0:.4) -- (4.9,.5) arc(0:-180:.4) -- (4.1,1.1) -- (2.5,1.1) arc(270:180:.4) -- (2.1,3.1) -- (.5,3.1) arc(270:180:.4) -- (.1,4.5);
\foreach \x in {1,2,3,4} {\draw (\x,0) -- (\x,5);}
\foreach \y in {1,2,3,4} {\draw (0,\y) -- (5,\y);}
\draw (-3,2) node[left] {$1$};
\draw (8,2) node {\phantom{1}};
\path[->] (-3,2.25) edge[bend left] (.5,3.5);
\fill[black] (1,3) circle (2mm);
\fill[black] (2,4) circle (2mm);
\fill[black] (3,1) circle (2mm);
\fill[black] (4,2) circle (2mm);
\end{tikzpicture}
\end{center}
where the marking of this region is $1$, as indicated.  The reader is referred to \cite{ulfarsson} for more information about these objects.

To suggest the relevance of the set $\Phi$, let us compare $\rep(\phi)$ and $\patt(\phi)$ for all $\phi \in \Phi$, writing $\rep(\phi)$ as the difference $\ell(\phi) - |\supp(\phi)|$, and $\patt(\phi)$ as the sum $\p(\phi) + \q(\phi)$ in equation~\eqref{eqn:patt definition}.
$$\begin{array}{c|c|c}
\phi \in \Phi & \rep(\phi) & \patt(\phi)\\
\hline
4321 &  6 - 3 = 3 & 4 + 0 = 4\\
34512 & 6 - 4 = 2 & 0 + 3 = 3\\
45123 & 6 - 4 = 2 & 0 + 3 = 3\\
35412 & 7 - 4 = 3 & 2 + 2 = 4\\
43512 & 7 - 4 = 3 & 2 + 2 = 4\\
45132 & 7 - 4 = 3 & 2 + 2 = 4\\
45213 & 7 - 4 = 3 & 2 + 2 = 4\\
53412 & 8 - 4 = 4 & 4 + 1 = 5\\
45312 & 8 - 4 = 4 & 4 + 1 = 5\\
45231 & 8 - 4 = 4 & 4 + 1 = 5
\end{array}$$
Observe that for each $\phi \in \Phi$, we have $\rep(\phi) < \patt(\phi)$.

We are now able to state the main theorem of the paper.

\begin{thm}\label{thm:main}
If a permutation $w$ avoids every pattern in the set $\Phi$, then
$$\rep(w) = \patt(w).$$
Otherwise,
$$\rep(w) < \patt(w).$$
\end{thm}

This characterization of equality between $\rep$ and $\patt$ if and only if the set $\Phi$ is avoided is recorded in entry P0022 of the Database of Permutation Pattern Avoidance \cite{dppa}, and is enumerated by A191721 in \cite{oeis}.

Observe that Theorem~\ref{thm:main} recovers the result in Theorem~\ref{thm:0 and 1}, since a permutation $w$ in which $\patt(w) \in \{0,1\}$ necessarily avoids every pattern in $\Phi$.  Note also that $0$ and $1$ are the only values for which $\rep(w)$ and $\patt(w)$ are always equal, because there are permutations $\phi \in \Phi$ with $\rep(\phi) = 2$ but $\patt(\phi) = 3$.

Suppose $w \in \mf{S}_N$.  Theorem~\ref{thm:main} is proved by induction on $N$ and involves an assignment of at least one $N$-occurrence of $321$ or $3412$ to each previously used letter involved in positioning $N$ in the one-line notation of $w$, after first positioning all other letters relative to each other.  We must be wary of overcounting these $N$-occurrences of $321$ and $3412$.  The details of the proof are covered in Sections~\ref{section:preliminaries} and~\ref{section:proof}.

\section{Preliminaries for proving the main theorem}\label{section:preliminaries}

\subsection{Notation and elementary results to be used in the proof}

\begin{sdefn}
Consider $w \in \mf{S}_N$.  Define $\ol{w} \in \mf{S}_{N-1}$ by
$$\ol{w}(i) = \begin{cases}
w(i) & \text{if } i < w^{-1}(N), \text{ and }\\
w(i+1) & \text{if } i > w^{-1}(N).
\end{cases}$$
The one-line notation of $\ol{w}$ is obtained from the one-line notation of $w$ by deleting the letter $N$ and sliding all subsequent letters one space to the left.  Moreover, if we think of $\ol{w}$ as a permutation in $\mf{S}_N$ that fixes $N$, then
\begin{equation}\label{eqn:moving N into place}
w = \ol{w}s_{N-1}s_{N-2}\cdots s_{w^{-1}(N)},
\end{equation}
and
$$\ell(w) = \ell(\ol{w}) + N - w^{-1}(N).$$
\end{sdefn}

\begin{sex}
If $w = 35412$, then $\ol{w} = 3412$.  If we consider $\ol{w}$ to be the element $34125 \in \mf{S}_5$, then
$$w = \ol{w}s_4s_3s_2.$$
One reduced decomposition of $\ol{w}$ is $s_2s_1s_3s_2$, and so $s_2s_1s_3s_2s_4s_3s_2$ is a reduced decomposition of $w$.
\end{sex}

Throughout the rest of this section, let $w$ be a permutation in $\mf{S}_N$, and $\ol{w} \in \mf{S}_{N-1}$ be as defined above.

The following set will be crucial in the proof of Theorem~\ref{thm:main}, describing the letters in a reduced word of $w$, but not of $\ol{w}$, which count as repeated letters for $w$.

\begin{sdefn}\label{defn:new-rep}
Let $\repeat(w) = \{k : s_k \in \supp(\ol{w}) \text{ and } w^{-1}(N) \le k\}$.
\end{sdefn}

\begin{slem}\label{lem:new repeats}
\begin{eqnarray*}
\rep(w) &=&\rep(\ol{w}) + |\supp(\ol{w}) \cap \{s_{N-1}s_{N-2}\cdots s_{w^{-1}(N)}\}|\\
&=& \rep(\ol{w}) + |\repeat(w)|.
\end{eqnarray*}
\end{slem}

\begin{proof}
This follows from equation~\eqref{eqn:moving N into place}.
\end{proof}

Recall the functions $M_k$ and $m_k$ from Definition~\ref{defn:max and min}.
 
\begin{slem}\label{lem:when in new-rep}
$M_k(\ol{w}) > m_k(\ol{w})$ and $w^{-1}(N) \le k$ if and only if $k \in \repeat(w)$.
\end{slem}

\begin{proof}
The forward direction of the statement is immediate from the definition of $\repeat(w)$.  The converse of this follows from Lemma~\ref{lem:support implications} and Definition~\ref{defn:new-rep}.
\end{proof}

To show that $\rep(w)$ is a lower bound for $\patt(w)$, we would like to assign, to each element of $\repeat(w)$, at least one $N$-occurrence in $w$ of one of the patterns $\{321,3412\}$.  This assignment should be done carefully to avoid overcounting.  Additionally, to characterize when $\rep(w)$ and $\patt(w)$ are equal, we would like to understand when each $N$-occurrence in $w$ of the patterns $\{321, 3412\}$ corresponds to some element of $\repeat(w)$.

For the remainder of this section, set $M_k = M_k(\ol{w})$ and $m_k = m_k(\ol{w})$ for all $k$.

\begin{sdefn}\label{defn:patterns}
Consider $k \in \repeat(w)$.  Define $\pattern_k(w)$ as follows.
\begin{enumerate}\renewcommand{\labelenumi}{\Roman{enumi}.}
\item If $w^{-1}(N) < w^{-1}(M_k)$, then $\pattern_k(w) = \{N,M_k,m_k\}$, which is a $321$-pattern in $w$.
\item If $w^{-1}(N) > w^{-1}(M_k)$ and $\ol{w}(k) > m_k$, then $\pattern_k(w) = \{N, \ol{w}(k), m_k\}$, which is a $321$-pattern in $w$.
\item Otherwise, set $\pattern_k(w) = \{M_k, N, \ol{w}(k), m_k\}$, which is a $3412$-pattern in $w$.
\end{enumerate}
This $\pattern_k(w)$ is undefined if $k \not\in \repeat(w)$.
\end{sdefn}

Note that $\pattern_k(w)$ is always an $N$-occurrence of either $321$ or of $3412$ because $k \in \repeat(w)$ and thus $M_k > m_k$ by Lemma~\ref{lem:support implications}.  However, it is not clear when $\pattern_k(w)$ and $\pattern_{k'}(w)$ coincide for $k \neq k'$, nor which $N$-occurrences of $321$ or of $3412$ have the form $\pattern_k(w)$ for some $k$.

\subsection{Issues of overcounting}

Consider whether the patterns $\pattern_k(w)$ might overcount $N$-occurrences of $321$ or $3412$ in $w$.

\begin{sprop}\label{prop:overcounting 3412}
There are no distinct $k, k' \in \repeat(w)$ for which $\pattern_k(w)$ and $\pattern_{k'}(w)$ are the same $N$-occurrence of $3412$ in $w$.
\end{sprop}

\begin{proof}
If this were the case, then $(M_k,\ol{w}(k),m_k) = (M_{k'},\ol{w}(k'),m_{k'})$.  But then $\ol{w}(k) = \ol{w}(k')$, implying that $k = k'$.
\end{proof}

Therefore, if there is any overcounting of $N$-occurrences of $321$ or $3412$ among the $\{\pattern_k(w)\}$, it must be that $\pattern_k(w)$ and $\pattern_{k'}(w)$ are the same $N$-occurrence of $321$.

\begin{sprop}\label{prop:overcounting M and m}
If there exist distinct $k, k' \in \repeat(w)$ with $\pattern_k(w) = \pattern_{k'}(w)$, then $w$ has an $N$-occurrence of $4321$.
\end{sprop}

\begin{proof}
Suppose that there exist $k, k' \in \repeat(w)$, with $k < k'$, such that $\pattern_k(w) = \pattern_{k'}(w)$.  Proposition~\ref{prop:overcounting 3412} implies that these coincident patterns must be $N$-occurrences of $321$ in $w$.

These coincident $\pattern_k(w)$ and $\pattern_{k'}(w)$ cannot both be of type II as in Definition~\ref{defn:patterns}, because that would mean that $\ol{w}(k) = \ol{w}(k')$, and so $k = k'$.

Now suppose that the patterns have different types.  Thus $\ol{w}(k) = M_{k'}$ and $m_k = m_{k'}$.  Then $\{N, \ol{w}(k) = M_{k'}, \ol{w}(k'), m_k = m_{k'}\}$ forms an $N$-occurrence of $4321$ in $w$.  Note also in this case that we must have $M_k = \ol{w}(k) = M_{k'}$, since otherwise $M_k$ would lie to the left of $\ol{w}(k) = M_{k'}$, and be greater than $\ol{w}(k)$ by definition, which would contradict the maximality of $M_{k'}$.

It remains to consider the case when both patterns are of type I, and so $(M_k, m_k) = (M_{k'}, m_{k'}) = (M,m)$.  We can assume that $M \not\in \{\ol{w}(k), \ol{w}(k')\}$ because that case was already addressed.  Then the one-line notation of $\ol{w}$, and hence of $w$, looks like
$$\cdots \ M \ \cdots \ \ol{w}(k) \ \cdots \ \ol{w}(k') \ \cdots \ m \ \cdots.$$
Consider where $N$ lies in relation to the values $\{M,\ol{w}(k),\ol{w}(k'),m\}$.  Because both patterns have type I, we must have that $w^{-1}(N) < w^{-1}(M)$, and so $N$ lies to the left of $M$.  The definitions of $M$ and $m$ require that $M > \ol{w}(k), \ol{w}(k')$, and $m < \ol{w}(k')$.  Thus the letters $\{N,M, \ol{w}(k'), m\}$ form an $N$-occurrence of $4321$.
\end{proof}

\begin{scor}\label{cor:overcounting}
If $w$ has no $N$-occurrence of $4321$, then 
$$|\repeat(w)| \le \patt_N(w).$$
\end{scor}

\begin{proof}
This follows from Propositions~\ref{prop:overcounting 3412} and~\ref{prop:overcounting M and m}, because there do not exist distinct $k, k' \in \repeat(w)$ with $\pattern_k(w)$ equalling $\pattern_{k'}(w)$.
\end{proof}

Therefore, by Corollary~\ref{cor:overcounting}, the procedure for assigning to each $k \in \repeat(w)$ an $N$-occurrence of either $321$ or $3412$ is injective if $w$ has no $N$-occurrence of $4321$.  We must now consider what happens to this assignment when $w$ does have such a pattern.

\begin{sprop}\label{prop:overcounting with 4321}
Suppose that $w$ has an $N$-occurrence of $4321$.
\begin{enumerate}\renewcommand{\labelenumi}{(\alph{enumi})}
\item If there exist distinct $k, k' \in \repeat(w)$ and $\pattern_k(w) = \pattern_{k'}(w)$, then there are two other $N$-occurrences $\pattern_{k'}^+(w)$ and $\pattern_{k'}^-(w)$ of $321$ in $w$, which are not equal to $\pattern_j(w)$ for any $j$.  (Let such a $k'$ be called ``duplicating.'')
\item Let $i$ and $j$ both be duplicating.  If $i \neq j$, then $\{\pattern_i^+(w), \pattern_i^-(w)\} \cap \{\pattern_{j}^+(w), \pattern_{j}^-(w)\} = \emptyset$.
\end{enumerate}
\end{sprop}

\begin{proof}
First we will prove statement (a).  Suppose that there are such $k<k'$.  Then we know from Proposition~\ref{prop:overcounting M and m} that $(M_k,m_k) = (M_{k'},m_{k'}) = (M,m)$, and $\pattern_k(w) = \pattern_{k'}(w) = \{N,M,m\}$.  Also, we know that the one-line notation of $w$ looks like
$$\cdots \ N \ \cdots \ M \ \cdots \ \ol{w}(k) \ \cdots \ \ol{w}(k') \ \cdots \ m \ \cdots,$$
where $M$ and $\ol{w}(k)$ could possibly be equal.  Because $M = M_{k'}$, we must have $M > \ol{w}(k')$.  Also, because $m = m_{k}$, we must have $m < \ol{w}(k')$.  Thus
\begin{equation}\label{eqn:+ and - patterns}
\pattern_{k'}^+(w) = \{N > M > \ol{w}(k')\} \text{ and } \pattern_{k'}^-(w) = \{N > \ol{w}(k') > m\}
\end{equation}
are both $N$-occurrences of $321$ in $w$.

Note that $\pattern_{k'}^+(w)$ is not equal to $\pattern_j(w)$ for any $j$, because $\ol{w}(k')$ is not equal to $m_j$ for any $j$: there exists a letter (for example, $m$) to the right of $\ol{w}(k')$ which is less than $\ol{w}(k')$.  Similarly, $\pattern_{k'}^-(w) \neq \pattern_{j}(w)$ for any $j$, because $\ol{w}(k')$ cannot equal $M_j$.

The proof of statement (b) is similar to the previous argument.  Suppose that $i$ and $j$ are duplicating, with $i \neq j$.  If $\pattern_i^+(w) = \pattern_j^+(w)$ or $\pattern_i^-(w) = \pattern_j^-(w)$, as defined in equation~\eqref{eqn:+ and - patterns}, then $\overline{w}(i) = \overline{w}(j)$.  This would mean that $i=j$, which is a contradiction.  Thus it remains to consider the situation $\pattern_i^+(w) = \pattern_j^-(w)$.  Then $M_i = \overline{w}(j)$ and $\overline{w}(i) = m_j$, and $i>j$.  Once again, we cannot have $\overline{w}(i) = m_j$, because the letter $m_i$ appears to the right of $\overline{w}(i)$ and is less than $\overline{w}(i)$.  This completes the proof.
\end{proof}

\begin{scor}\label{cor:strict overcounting}
If $w$ has an $N$-occurrence of $4321$, then
$$|\repeat(w)| < \patt_N(w).$$
\end{scor}

\begin{proof}
Partition the set $\repeat(w)$ into sets $S_1, S_2, \ldots, S_t$ so that $(M_k,m_k) = (M(i),m(i))$ for each $k \in S_i$.  Suppose $S_i = \{k_{i_1} < k_{i_2} < \cdots < k_{i_{|S_i|}}\}$, and define
$$\pattern_{S_i}(w) = \left\{\pattern_{k_{i_1}}(w), \pattern_{k_{i_2}}^+(w), \pattern_{k_{i_2}}^-(w), \ldots, \pattern_{k_{i_{|S_i|}}}^+(w), \pattern_{k_{i_{|S_i|}}}^-(w)\right\}.$$
Note that if $|S_i| = 1$, then $|\pattern_{S_i}(w)| = 1$.  Also, if $|S_i| > 1$, then $|\pattern_{S_i}(w)| = 2|S_i| - 1 > |S_i|$.  Moreover, the elements of $\pattern_{S_i}(w)$ are all $N$-occurrences of either $321$ or $3412$ in $w$.  Finally, by Proposition~\ref{prop:overcounting with 4321}(b), the sets $\{\pattern_{S_i}(w)\}$ are disjoint.

If $w$ has an $N$-occurrence of $4321$, then there exists some $S_i$ containing at least two elements.  Thus $|\repeat(w)| = |S_1| + |S_2| + \cdots + |S_t| < |\pattern_{S_1}(w)| + |\pattern_{S_2}(w)| + \cdots +  |\pattern_{S_t}(w)|\le \patt_N(w)$.
\end{proof}

Using the notation from the proof of Corollary~\ref{cor:strict overcounting}, we can also rewrite its result to say that the map
\begin{equation}\label{eqn:xi}
\xi_n: j \mapsto
\begin{cases}
\pattern_j(w) & \text{if } j \text{ is the minimal element in } S_i\text{, and}\\
\pattern_j^+(w) & \text{ otherwise.}
\end{cases}
\end{equation}
is an injection.

\subsection{Issues of undercounting}

We have now addressed the issue of whether the set $\{\pattern_k(w) : k \in \repeat(w)\}$ might overcount some $N$-occurrences of $321$ or of $3412$ (in fact, we have shown that only $321$-patterns may be overcounted).  We must now consider when this set might undercount these $N$-occurrences.  As we have seen in Proposition~\ref{prop:overcounting with 4321}, undercounting is certainly a possibility.  What we will show now is that if $w$ avoids the ten patterns in the set $\Phi$, then there is no undercounting, and thus the inequality of Corollary~\ref{cor:overcounting} is actually an equality.  

To examine potential undercounting, we must decide if and when an $N$-occurrence of $321$ or of $3412$ might not equal $\pattern_k(w)$ for some $k$.

\begin{sprop}\label{prop:321 must use m_k}
If any $N$-occurrence $\{N>a>b\}$ of $321$ in $w$ is such that $b \not\in \{m_k: k \in \repeat(w)\}$, then $w$ has an $N$-occurrence of $4321$.
\end{sprop}

\begin{proof}
Suppose there is an $N$-occurrence of $321$ in $w$ where $b \neq m_k$ for any $k$.  Then to the right of $b$ in the one-line notation $w$, there exists $c < b$, preventing $b$ from equalling any such $m_k$.  Thus $\{N>a>b>c\}$ is an $N$-occurrence of $4321$.  Now, suppose there is no such $c$, and set $k = \ol{w}^{-1}(b) - 1$.  Then $b = m_k$, and, since $a > b$ appears to the left of $b$, we must have $M_k \ge a > b=m_k$.  Therefore, by Lemma~\ref{lem:when in new-rep}, $k \in \repeat(w)$.
\end{proof}

\begin{sprop}\label{prop:321 must use M_k or ol{w}(k)}
Suppose $w$ is $4321$-avoiding.  If any $N$-occurrence $\{N> a > b\}$ of $321$ in $w$ is such that there exists no $k \in \repeat(w)$ with $(a,b) \in \{(M_k,m_k), (\ol{w}_k,m_k)\}$, then $w$ has an $N$-occurrence of at least one of the patterns $\{45312, 53412\}$.
\end{sprop}

\begin{proof}
By Proposition~\ref{prop:321 must use m_k}, we know that $b = m_k$ for at least one value of $k \in \repeat(w)$.  Suppose that $a \not\in \{M_k,\ol{w}(k)\}$.

Suppose $a > M_k$.  Then, by maximality of $M_k$, this $a$ must appear to the right of both $M_k$ and $\ol{w}(k)$ in the one-line notation of $w$.  But then, setting $k' = \ol{w}^{-1}(a)$, we must have $m_{k'} = m_k = b$, and so $(a,b) = (\ol{w}(k'), m_{k'})$.  By definition, $k' > k$, and so $m_{k'} = m_k < k+1 < k'+1$, where the first inequality is because $k \in \repeat(w)$.  Therefore $k' \in \repeat(w)$ as well, by Lemma~\ref{lem:when in new-rep}.

Now suppose that $a < M_k$.  If $w^{-1}(N) > w^{-1}(M_k)$, then the one-line notation of $w$ looks like
$$\cdots \ M_k \ \cdots \ N \ \underbrace{\cdots\cdots}_{<a} \ a \underbrace{\cdots\cdots}_{>a \text{ or } < m_k} b=m_k \ \cdots,$$
because $w$ is $4321$-avoiding.  If all values appearing between $a$ and $b$ in $w$ are larger than $a$, then we can set $k' = \ol{w}^{-1}(a)$, and we have $(M_{k'},m_{k'}) = (M_k,m_k)$, and again $k' \in \repeat(w)$ by Lemma~\ref{lem:when in new-rep}.  Thus suppose that there is some value $c$ in this portion of $w$ with $c < m_k$.  Then $\{M_k, N, a, c, b\}$ is an $N$-occurrence of $45312$ in $w$.

Finally, suppose that $w^{-1}(N) < w^{-1}(M_k)$, where $k$ is minimal with this property.  So the one-line notation of $w$ looks like
$$\cdots \ N \ \cdots \ a \ \cdots \ M_k \ \cdots \ b=m_k\ \cdots,$$
again because $w$ is $4321$-avoiding.  If $M_k = \overline{w}(k)$, then the value of $k$ was not chosen to be minimal, a contradiction.  Thus the entry $\overline{w}(k)$ must lie strictly between $M_k$ and $b=m_k$.  By definition, $\overline{w}(k) < M_k$.  Moreover, to avoid the pattern $4321$, we must have $\overline{w}(k) < b=m_k$.  Thus $\{N, a, M_k, \overline{w}(k), b=m_k\}$ forms a $53412$-pattern in $w$.
\end{proof}

Propositions~\ref{prop:321 must use m_k} and~\ref{prop:321 must use M_k or ol{w}(k)} now imply the following result.

\begin{scor}\label{cor:321 patterns}
If $w$ has no $N$-occurrences of the patterns $\{4321, 45312, 53412\}$, then every $N$-occurrence of $321$ in $w$ is equal to $\pattern_k(w)$ for some $k$.
\end{scor}

\begin{sprop}\label{prop:3412 must use m_k}
If any $N$-occurrence $\{a,N,b,c\}$ of $3412$ in $w$ is such that $c \not\in \{m_k : k \in \repeat(w)\}$, then $w$ has an $N$-occurrence of at least one of the patterns $\{45231, 45132\}$.
\end{sprop}

\begin{proof}
Suppose there is such an $N$-occurrence of $3412$ in $w$.  This means that to the right of $c$ in the one-line notation of $w$, there exists a $d < c$, preventing $c$ from equalling any such $m_k$.  Thus $\{a, N, b, c, d\}$ is an $N$-occurrence of either $45231$ or of $45132$, depending on whether $b>d$ or $b<d$.  Now suppose that there is no such $d$, and set $k = \ol{w}^{-1}(c)-1$.  Then $c = m_k$, and, since $a > c$ appears to the left of $c$, we must have $M_k > m_k$.  Therefore, by Lemma~\ref{lem:when in new-rep}, $k \in \repeat(w)$.
\end{proof}

\begin{sprop}\label{prop:3412 must use M_k}
Suppose $w$ is $45231$- and $45132$-avoiding.  If any $N$-occurrence $\{a,N,b,c\}$ of $3412$ in $w$ is such that there exists no $k \in \repeat(w)$ with $(a,c) = (M_k,m_k)$, then $w$ has an $N$-occurrence of at least one of the patterns $\{43512, 34512, 35412\}$.
\end{sprop}

\begin{proof}
By Proposition~\ref{prop:3412 must use m_k}, we know that $c = m_k$ for some $k \in \repeat(w)$.  Choose the minimal such $k$; that is, choose $k$ so that $\ol{w}(k) < c$ (and thus, necessarily, $\ol{w}(j) \ge c$ for all $j > k+1$).  There are now three places $M_k$ might appear relative to the letters $\{a,N,\ol{w}(k),c\}$, which themselves form an $N$-occurrence of $3412$ in $w$:
$$\underbrace{\cdots\cdots}_{M_k?} \ a \ \underbrace{\cdots\cdots}_{M_k?} \ N \ \underbrace{\cdots\cdots}_{M_k?} \ \ol{w}(k) \ \cdots \ c=m_k\ \cdots.$$
By definition, $M_k \ge a$.  Thus, if $M_k \neq a$, then these three possibilities create $N$-occurrences of $43512$, $34512$, or $35412$ in $w$, respectively.
\end{proof}

\begin{sprop}\label{prop:3412 unique 1}
Suppose $w$ avoids the patterns
$$\{45231, 45132, 43512, 34512, 35412\}.$$
If any $N$-occurrence $\{a,N,b,c\}$ of $3412$ in $w$ is such that there exists no $k \in \repeat(w)$ with $(a,b,c) = (M_k, \ol{w}(k), m_k)$, then $w$ has an $N$-occurrence of at least one of the patterns $\{45123, 45213\}$.
\end{sprop}

\begin{proof}
By Propositions~\ref{prop:3412 must use m_k} and~\ref{prop:3412 must use M_k}, we know that $(a,c) = (M_k,m_k)$ for some $k \in \repeat(w)$.  If $b \neq \ol{w}(k)$, then $\ol{w}(k)$ either lies between $N$ and $b$, or between $b$ and $c = m_k$.  In fact, $\ol{w}(k)$ must lie to the right of $b$, because $b < c = m_k = \min\{\ol{w}(k+1), \ldots, \ol{w}(n)\}$.  We also know that $\ol{w}(k) < M_k = a$.  Therefore, since $w$ is $45132$-avoiding, the set $\{a,N,b,\ol{w}(k),c\}$ forms an $N$-occurrence of either $45123$ or $45213$.
\end{proof}

Propositions~\ref{prop:3412 must use m_k}, \ref{prop:3412 must use M_k}, and~\ref{prop:3412 unique 1} now imply the following result.

\begin{scor}\label{cor:3412 patterns}
If $w$ has no $N$-occurrences of the patterns
$$\{45231, 45132, 43512, 34512, 35412, 45123, 45213\},$$
then every $N$-occurrence of $3412$ in $w$ is equal to $\pattern_k(w)$ for some $k \in \repeat(w)$.
\end{scor}

This addresses the concern about undercounting the $N$-occurrences of $321$ and $3412$ in $w$.

\begin{scor}\label{cor:undercounting}
If $w$ has no $N$-occurrence of any of the patterns in the set
$$\{4321, 45312, 53412, 45231, 45132, 43512, 34512, 35412, 45123, 45213\},$$
then
$$|\repeat(w)| \ge \patt_N(w).$$
\end{scor}

\begin{proof}
This follows from Corollaries~\ref{cor:321 patterns} and~\ref{cor:3412 patterns}.
\end{proof}

\subsection{Conclusions}

We now combine the previous two subsections to draw the following conclusion.

\begin{scor}\label{cor:avoiding Phi is equality}
If $w$ has no $N$-occurrence of any of the patterns in the set $\Phi$, then
$$|\repeat(w)| = \patt_N(w).$$
In other words, if $w$ has no $N$-occurrence of any of the patterns in the set $\Phi$, then the map $\xi_N$ of equation~\eqref{eqn:xi} is a bijection.
\end{scor}

\begin{proof}
Combine the inequalities in Corollaries~\ref{cor:overcounting} and~\ref{cor:undercounting}.
\end{proof}

It is natural now to wonder about the implications of containing an $N$-occurrence of a pattern in $\Phi$.  In fact, for each $w$ containing an $N$-occurrence of some $\phi \in \Phi$, there is an $N$-occurrence $\pattern_{\phi}(w)$ of either $321$ or $3412$ which is not equal to $\pattern_k(w)$ or to $\pattern_k^+(w)$ (as defined in Proposition~\ref{prop:overcounting with 4321}) for any $k$, as is shown in the following table.  In this table, the $N$-occurrence $\pattern_{\phi}(w)$ will be written as a substring of $\phi$, and will refer to those respective letters of the $N$-occurrence of $\phi$ in $w$.

$$\begin{array}{l|l}
\phi \in \Phi & \pattern_{\phi}(w)\\
\hline
4321 & 421\\
34512 & 3512\\
45123 & 4513\\
35412 & 3512\\
43512 & 3512\\
45132 & 4513\\
45213 & 4523\\
53412 & 532\\
45312 & 532\\
45231 & 4523
\end{array}$$
Note that for $4321 \in \Phi$, the subpattern $432$ is also not equal to any $\pattern_k(w)$.  However, it could equal some $\pattern_{k}^+(w)$, so to avoid this possibility we set $\pattern_{4321}(w) = 421$.

\begin{sprop}
Let $w \in \mf{S}_N$ be a permutation containing an $N$-occurrence of some pattern $\phi \in \Phi$.  Then $\pattern_{\phi}(w)$ is not equal to $\pattern_k(w)$ for any $k \in \repeat(w)$, nor to any $\pattern_{k}^+(w)$, as defined in Proposition~\ref{prop:overcounting with 4321}.  That is, the injection $\xi_N$ of equation~\eqref{eqn:xi} is not surjective.
\end{sprop}

\begin{proof}
This follows from the definitions of the patterns $\pattern_{\phi}(w)$, $\pattern_k(w)$, and $\pattern_{k}^+(w)$.
\end{proof}

This proposition has the following corollary.

\begin{scor}\label{cor:having Phi is inequality}
If $w$ has an $N$-occurrence of at least one of the patterns in the set $\Phi$, then
$$|\repeat(w)| < \patt_N(w).$$
\end{scor}

\section{Proof of the main theorem}\label{section:proof}

\begin{proof}[Proof of Theorem~\ref{thm:main}]
We prove this by induction on the number of letters in a permutation.

The result is easy to verify for small cases, so assume that the theorem holds for all permutations in $\mf{S}_n$ for all $n < N$, and consider $w \in \mf{S}_N$.  Define $\ol{w} \in \mf{S}_{N-1}$ as in Section~\ref{section:preliminaries}.  Since $N-1 < N$, we know that $\rep(\ol{w})$ is equal to $\patt(\ol{w})$ if $\ol{w}$ avoids the patterns in the set $\Phi$, and that $\rep(\ol{w})$ is less than $\patt(\ol{w})$ if $\ol{w}$ contains at least one pattern in $\Phi$.

Suppose first that $\ol{w}$ avoids the patterns in $\Phi$.  If $w$ has no $N$-occurrences of any of the patterns in $\Phi$, then $|\repeat(w)| = \patt_N(w)$.  Thus
\begin{eqnarray*}
\rep(w) &=& \rep(\ol{w}) + |\repeat(w)|\\
&=& \patt(\ol{w}) + \patt_N(w) = \patt(w).
\end{eqnarray*}
On the other hand, if $w$ does have an $N$-occurrence of at least one the patterns in $\Phi$, then $|\repeat(w)| < \patt_N(w)$, and so
\begin{equation}\begin{split}\label{ineq:result}
\rep(w) &= \rep(\ol{w}) + |\repeat(w)|\\
&< \patt(\ol{w}) + \patt_N(w) = \patt(w).
\end{split}\end{equation}

Now assume that $\ol{w}$ does not avoid the patterns in $\Phi$.  If $w$ has no $N$-occurrences of any of the patterns in $\Phi$, then $|\repeat(w)| = \patt_N(w)$.  Thus inequality \eqref{ineq:result} holds.  On the other hand, if $w$ does have an $N$-occurrence of at least one the patterns in $\Phi$, then $|\repeat(w)| < \patt_N(w)$, and so inequality \eqref{ineq:result} holds again.

This completes the proof.
\end{proof}

\begin{defn}
Consider a permutation $w \in \mf{S}_N$.  Let $\ol{w}_{(0)} = w$, and for $i \in \{1, \ldots, N-1\}$, let $\ol{w}_{(i+1)} = \overline{\ol{w}_{(i)}}$.
\end{defn}

\begin{cor}
If a permutation $w \in \mf{S}_N$ avoids every pattern in the set $\Phi$, then the maps $\{\xi_n : n \le N\}$ define a bijection from the set $\{\repeat(\ol{w}_{(i)}): i \in \{0, \ldots, N-1\}\}$ to the set of all $321$- and $3412$-patterns in $w$.
\end{cor}

Additionally, the proof of Theorem~\ref{thm:main} can be adapted to show the following.

\begin{cor}\label{cor:bound}
For any permutation $w$,
$$\patt(w) - \rep(w) \ge |\{r : w \text{ has an } r\text{-occurrence of a pattern in }\Phi\}|.$$
\end{cor}

Using \cite{albert} and the MAPLE package \cite{vatter}, Vince Vatter has subsequently found a generating function for the number of permutations in $\mf{S}_N$ avoiding the ten patterns in $\Phi$ \cite{vatter-email}.  This generating function is
$$g(x) = \frac{1-4x+x^3}{(1-x)(1-4x-x^2+x^3)}.$$

\section*{Acknowledgements}

Special thanks are due to K\'ari Ragnarsson for writing the Java code to compute examples and to confirm conjectures through $\mf{S}_{10}$, as well as to an anonymous referee for many thoughtful comments.


\begin{thebibliography}{99}

\bibitem{albert} M. H. Albert, S. Linton, N. Ru\v{s}kuc, The insertion encoding of permutations, {Electron. J. Combin.} {12} (2005), R47.

\bibitem{bjorner-brenti} A.~Bj\"{o}rner, F.~Brenti, {Combinatorics of Coxeter Groups}, Graduate Texts in Mathematics 231, Springer, New York, 2005.

\bibitem{claesson-kitaev} A.~Claesson, S.~Kitaev, Classification of bijections between $321$- and $132$-avoiding permutations, {S\'em.~Lothar.~Combin.} {60} (2008), B60d.

\bibitem{daly} D. Daly, Reduced decompositions with one repetition and permutation pattern avoidance, preprint.

\bibitem{macdonald} I.~G.~Macdonald, {Notes on Schubert Polynomials}, Laboratoire de combinatoire et d'informatique math\'{e}matique (LACIM), Universit\'{e} du Qu\'{e}bec \`{a} Montr\'{e}al, Montreal, 1991.

\bibitem{marcus-tardos} A.~Marcus, G.~Tardos, Excluded permutation matrices and the Stanley-Wilf conjecture, {J.~Combin.~Theory Ser.~A} {107} (2004), 153--160.

\bibitem{matsumoto} H.~Matsumoto, G\'en\'erateurs et relations des groupes de Weyl g\'en\'eralis\'es, {C.~R.~Acad.~Sci.~Paris} {258} (1964), 3419--3422.

\bibitem{simion-schmidt} R.~Simion, F.~W.~Schmidt, Restricted permutations, {European J.~Combin.} {6} (1985), 383-406.

\bibitem{oeis} N.~J.~A.~Sloane, Online encyclopedia of integer sequences, published electronically at {http://oeis.org}.

\bibitem{stanley} R.~P.~Stanley, On the number of reduced decompositions of elements of Coxeter groups, {European J.~Combin.} {5} (1984), 359--372.

\bibitem{dppa} B. E. Tenner, Database of permutation pattern avoidance, published electronically at  \hfill \phantom{*} {http://math.depaul.edu/bridget/patterns.html}. 

\bibitem{patt-bru} B. E. Tenner, Pattern avoidance and the Bruhat order, {J. Combin. Theory Ser. A} {114} (2007), 888--905.

\bibitem{rdpp} B. E. Tenner, Reduced decompositions and permutation patterns, {J. Algebraic Combin.} {24} (2006), 263--284.

\bibitem{tits} J.~Tits, Le probl\`eme des mots dans les groupes de Coxeter, {Symposia Mathematica (INDAM, Rome, 1967/68)}, Academic Press, London, 1969, vol.~1, pp.175--185.

\bibitem{ulfarsson} H. \'Ulfarsson, A unification of permutation patterns related to Schubert varieties, to appear in {Pure Math. Appl.}

\bibitem{vatter} V. Vatter, the MAPLE package {INSENC}, published electronically at \hfill \phantom{*} {http://www.math.ufl.edu/\~{}vatter/publications/ins-enc/}.

\bibitem{vatter-email} V. Vatter, private communication.


\end{thebibliography}
\end{document}